 \def\dsp{\displaystyle}
\newenvironment{proof}{{\noindent
\textbf{Proof}\,\,}}{\hspace*{\fill}$\Box$\medskip}

\newtheorem{Def}{Definition} \newtheorem{Thm}{Theorem}
 \newtheorem{Lem}{Lemma}
 \newtheorem{Rem}{Remark}
\newtheorem{Prop}{Proposition} \newtheorem{Cor}{Corollary}

\documentclass[12pt] {article}
\usepackage{amsfonts,graphicx,amsmath,amscd}

\def\zz{\mathbb{Z}}  \def\rr{\mathbb{R}}
\def\F{\mathcal F}  \def\G{\mathcal G}
 \def\B{\mathcal B} \def\g{\mathfrak g} \def\e
{\varepsilon }  \def\ph {\varphi } \def\b {\beta }
\def\Ga {\Gamma } \def\diffeo {diffeomorphism }   
\def\diffeos {diffeomorphisms } 
 \def\skpr {skew product } \def\skprs
{skew products } \def\st {such that } 
\def\man {manifold } \def\mans {manifolds } \def\om {\omega } \def\a
{\alpha } \def\skpr {skew product }  \def\l
{\lambda } \def\d {\delta }  \def\L {\Lambda }
   \def\om {\omega}    
  \def\La {\Lambda}
  \def\a{\alpha} \def\ho
{H\"{o}lder }  

\def\ga {\gamma}

\def\zz{\mathbb{Z}}  \def\rr{\mathbb{R}}
\def\F{\mathcal F} \def\G{\mathcal G} \def\B{\mathcal B}
\def\L{\mathcal L}    \def\e {\varepsilon }   \def\ph {\varphi } \def\g {\mathfrak g} \def\a {\alpha}
\def\b {\beta}  \def\ga {\gamma } \def\Ga {\Gamma }
\def\BB {\mathcal{B}} \def\diffeo {diffeomorphism }  \def\diffeos {diffeomorphisms }   \def\skpr
{skew product } \def\skprs {skew products } \def\st {such that}
 \def\man {manifold } \def\mans {manifolds }
\def\om {\omega } \def\skpr {skew product } 
\def\l {\lambda } \def\L {\Lambda }     \def\ho {H\"{o}lder } 

\def\st{such that }  \def\tes{there exists }
  \def\ka{\kappa}
  
\def\e{\varepsilon}   
  
 \def\mLip{\mbox {Lip }}
 \def\hd{Hausdorff dimension }

\def\hps{Hirsch-Pugh-Shub }

\title
   {}

\title{\ho properties of perturbed skew products and Fubini
regained}

\author{Yu. Ilyashenko\thanks{ The author was supported by part by the grants
NSF 0700973, RFBR 07-01-00017-à, RFFI-CNRS 050102801} \thanks
{Cornell University, US; Moscow State and Independent Universities,
Steklov Math. Institute, Moscow} and A. Negut \thanks{``Simion Stoilow''
 Institute of Mathematics of the Romanian Academy, Romania; Harvard University, US}}

\date{}

\begin{document}

\maketitle

\begin{abstract} In 2006 A. Gorodetski proved that central fibers of
perturbed skew products are \ho continuous with respect to the base
point. In the present paper we give an explicit estimate of the \ho
exponent mentioned above. Moreover, we extend the Gorodetski theorem
from the case when the fiber maps are close to the identity to a
much wider class that satisfy the so-called modified dominated splitting
condition. In many cases (for example, in the case of skew products
over the solenoid or over linear Anosov diffeomorphisms of a torus), the \ho exponent is
close to 1. This allows us in a sense to overcome the so called
Fubini nightmare. Namely, we prove that the union of central fibers
that are strongly atypical from the point of view of the ergodic
theory, has Lebesgue measure zero, despite the lack of absolute
continuity of the holonomy map for the central foliation. For that
we revisit the Hirsch-Pugh-Shub theory, and estimate the contraction
constant of the graph transform map.

\end{abstract}

\section{Introduction}
\label{sec:intro}

\subsection{Skew products and \ho continuity}\label{sub:hofub}

In this paper we study perturbations of \skprs over hyperbolic maps
in the base. Under the so-called dominated splitting condition,
these perturbations have an invariant center lamination (see
\cite{HPS}). It was proven by Gorodetski that this lamination will
be \ho continuous, \cite{G06} (see also \cite{N} for an earlier
particular result). In this paper we estimate the \ho exponent, and
prove that in some cases it can be made arbitrarily close to $1$ by
making the perturbation small enough. \\

Such a center lamination allows us to conjugate any perturbation of
a skew product with another skew product, which can be very useful
in applications. A priori, the conjugation is only a homeomorphism,
and as such it might not behave nicely with the measure. However,
the \ho property which we prove gives us some control over this
conjugation, and allows to resolve a number of measure-related
issues. \\

For example, the \ho property allows us to overcome in some sense
the Fubini nightmare. In more detail, we prove that perturbations of
a skew product over the Smale-Williams solenoid are semiconjugated
with the duplication of the circle. We prove that the fibers of this
semiconjugacy $q$ (i.e. the manifolds $q^{-1}(y), \ y \in S^1$) are
\ho in $y$ with exponent close to $1$. As a corollary, we prove that
for a set $A\subset S^1$ of points with ``strongly nonergodic
orbits'' under the duplication of the circle, the inverse image
$q^{-1}(A)$ has Lebesgue measure $0$ (Theorem~\ref{thm:fubini}
below). This is proved despite the fact that the foliation by the
fibers $q^{-1}(y), y \in S^1$ may not have an absolutely continuous
holonomy. \\

The entire Section~\ref{sec:intro} is concerned with statements of results.
In Subsection~\ref{sub:perhol} below, we introduce our Main Theorem~\ref{thm:mt}
for perturbations of skew products over arbitrary hyperbolic maps.
In the next Subsection~\ref{sub:st} we introduce Theorem~\ref{thm:st},
which improves the Main Theorem in some cases (an important example of which being the solenoid).
Two applications of Theorem~\ref{thm:st} are presented in Subsections~\ref{sub:fubini} and \ref{sub:kan}. \\

The sections after that will be concerned with proofs. In
Sections~\ref{sec:grtrans} and \ref{sec:hold} we work with
laminations and graph transform operators, culminating with the
proof of the Main Theorem~\ref{thm:mt}. Several of the results we
obtain help us in Section~\ref{sec:st}, where we prove
Theorem~\ref{thm:st}. In Section~\ref{sec:fubini}, we \emph{regain
the Fubini property of our central leaves}, thus proving
Theorem~\ref{thm:fubini}. Finally, Section~\ref{sec:app} consists of
an appendix where several technical results are proved.

\subsection{Persistence and \ho property for skew products}
\label{sub:perhol}

Throughout this paper a \emph{$C^r-$morphism} will refer to a $C^r$
map with a $C^r$ inverse. We will use this notion  both for maps of
a manifold (with or without boundary) onto itself, and for maps of a
manifold with boundary strictly into itself.  \\

Given a linear operator $A$ on a normed linear space $V$, when we
write $a\leq |A|\leq b$ we mean that $$ a|v|\leq |A(v)|\leq b|v|,
\textrm{ }\textrm{ }\textrm{ }\forall v\in V. $$ We will use this
convention repeatedly throughout the paper. \\

Let $B$ be a compact Riemannian manifold, henceforth called the
\emph{base}. Suppose $h:B\rightarrow B$ is a $C^2-$morphism with a
hyperbolic invariant subset $\Lambda \subset B$. When $h$ is onto,
we can take $B = \Lambda $. When $h$ is into, we can take $\Lambda $
to be the maximal attractor of $h$: $$\Lambda = \bigcap_{n \geq 0}
h^n(B).$$ Being hyperbolic, the map $h$ will have contracting and
expanding directions in the vicinity of $\Lambda$. Thus there exist
a Riemannian metric $d$ on $B$ and real numbers $0\leq \lambda_-\leq
\lambda<1$ and $0\leq \mu_-\leq \mu<1$, as well as a decomposition
of the tangent bundle: \begin{equation} \label{eqn:splittangent}
TB|_\Lambda=E^s\oplus E^u, \end{equation} such that $$dh:E^s
\rightarrow E^s\textrm{ }\textrm{ }\textrm{ and }\textrm{ }\textrm{
} \lambda_- \leq |dh|\leq \lambda, $$ \begin{equation}
\label{eqn:hypstructure} dh:E^u \rightarrow E^u\textrm{ }\textrm{
}\textrm{ and }\textrm{ }\textrm{ } \mu_-\leq |dh^{-1}|\leq \mu.
\end{equation} Note that if $\lambda_-=\mu_-=0$, then we get the
standard notion of hyperbolicity. \\

We assume that the bundles $E^s$
and $E^u$ are trivialized, i.e. that there exist isomorphisms over
$B$:
\begin{equation} \label{eqn:triv} \varphi^s:B \times \rr^{k}
\rightarrow E^s, \textrm{ }\textrm{ }\textrm{ } \varphi^u:B \times
\rr^l \rightarrow E^u \end{equation} for some positive integers
$k,l\geq 0$. The above is a technical condition necessary for our
proof, but we conjecture that all our results hold without it. Let
us note that it holds when $h$ is the Smale-Williams solenoid map or
any linear Anosov diffeomorphism of a torus. \\

\begin{Def} \label{def:lmhyp} An invariant set $\L$ of a map $h$
with the above properties will be called  $(\l_-, \l, \mu_-, \mu)$-
hyperbolic. \end{Def}

Take another compact manifold $M$, called the \emph{fiber}, and form
the Cartesian product $X=B\times M$. A \emph{skew product} over $h$
is defined as any $C^1-$map of the form

\begin{equation}
\label{eqn:skprod} \F:X\rightarrow X,\textrm{ }\textrm{ }\textrm{ }\F(b,m)=(h(b), f_b(m)),
\end{equation}
where $f_b(m):M \rightarrow M$ is a $C^1$ family of $C^1-$morphisms.

\begin{Def} \label{def:domsplit} We say that the skew product
\eqref{eqn:skprod} satisfies the \emph{ modified dominated splitting
condition} if

\begin{equation} \label{eqn:dom} \max\left(\max(\lambda,\mu)+
{\left| \left| \frac {\partial f_b^{\pm 1}}{\partial b}\right |
\right | }_{C^0(X)}, \ {\left| \left| \frac {\partial f_b^{\pm
1}}{\partial m}\right | \right | }_{C^0(X)}\right)=:L <
\min(\lambda^{-1}, \mu^{-1}).\end{equation} \end{Def}

Skew products are very useful in constructing dynamical systems with
various properties. However, one often wants to study generic
phenomena of dynamical systems, and therefore one also has to study
small perturbations of skew products.

\begin{Def}\label{def:rpert} Given $\rho>0$, a
\emph{$\rho-$perturbation} of the skew product \eqref{eqn:skprod} is
a $C^1-$morphism $\G:X \rightarrow X$ such that \begin{equation}
\label{eqn:perturbation} d(\G^{\pm 1}, \F^{\pm 1})_{C^1(X)}< \rho.
\end{equation} \end{Def}

Let us make a notational convention. In this paper, we will consider a fixed skew product $\F$ and a neighborhood $\Omega \ni \F$ in the $C^1-$norm. We will often be concerned with small perturbations $\G \in \Omega$ of $\F$, and with various geometric objects related to these
perturbations (such as central foliations, H\"older exponents etc). The leaves of
central foliations of the perturbed maps $\G$ are graphs of parameter dependent
maps $\beta_b$, or in other words, parameter dependent perturbations of the central fibers of $\F$. Whenever we write $||\beta_b|| = O(\rho)$, we mean that there
exists a constant $C$ depending only on $\Omega$, such that for any $\rho-$perturbation $\G \in \Omega$ the
maps corresponding to all central leaves satisfy the inequality $||\beta_b||\leq C\rho$. Thus the operator that maps $\G$ to $\beta_b$ is Lipshitz at $\F$ with constant $C$ (uniformly in $b$). We will consider other (parameter dependent)
operators and functionals defined on $\Omega$; the expression $O(\rho)$ has the same meaning for them. \\

Small perturbations of skew products are not necessarily skew
products anymore. However, in this paper we will show that they are
conjugated to skew products, and moreover the conjugation map
satisfies a \ho continuity property. \\

We will now state our main result.

\begin{Thm} [The Main Theorem] \label{thm:mt}  Consider a skew
product $\F $ as in \eqref{eqn:skprod} over a $(\l_-, \l, \mu_-, \mu
)-$ hyperbolic map in the base, satisfying the modified dominated
splitting condition. Then for small enough $\rho>0$, any
$\rho-$ perturbation $\G $ of $\F$ has the following properties: \\

a) There exists a $\G-$invariant set $Y\subset X$ and a continuous
map $p:Y\rightarrow B$ such that the diagram

\begin{equation} \label{eqn:semi}
\begin{CD} Y @>{\G}>> Y \\
@V{p}VV @V{p}VV \\
\Lambda @>{h}>> \Lambda
\end{CD}
\end{equation}
commutes. Moreover, the map
\begin{equation}
\label{eqn:homeo1}
H: Y \to \Lambda \times M, \textrm{ }\textrm{ }\textrm{ } H(b,m)=(p(b,m),m)
\end{equation}
is a homeomorphism. \\

b) The fibers $W_b = p^{-1}(b)$ are  Lipschitz close to vertical
(constant) fibers, and \ho continuous in $b$.  This means that $W_b$
is the graph of a Lipschitz map $\widetilde{\beta}_b:M \to B$ such that
\begin{equation}
\label{eqn:fibers}
d(\widetilde{\beta}_b,b)_{C^0} \leq O(\rho), \textrm{ }\textrm{ }\textrm{ }\mLip \widetilde{\beta}_b \leq O(\rho)
\end{equation}

\begin{equation}    \label{eqn:hol}
d(\widetilde{\beta}_b,\widetilde{\beta}_{b'})_{C^0} \le \frac {d(b , b' )^{\a-O(\rho)}}{O(\rho)^\alpha},
\end{equation}
where
\begin{equation}
\label{eqn:defalpha}
\alpha=\min\left(\frac {\ln \lambda}{\ln \lambda_-}, \frac {\ln
\mu}{\ln \mu_-}\right).
\end{equation}
Moreover, the map $H^{-1}$ is also \ho continuous, with the same $\alpha$.
\end{Thm}

\begin{Rem} Let us first make a remark about the exponent $\alpha$.
In many cases (e.g. when $h$ is the solenoid map or a linear Anosov
diffeomorphism of a torus), it may happen that $\lambda_-=\lambda$
and $\mu_-=\mu$. In that case, in the above theorem we have
$\alpha=1$, and thus the \ho exponent can be made arbitrarily close
to 1 by making $\rho$ small enough. \end{Rem}

\begin{Rem}
If $\Lambda=B$ (which would require $h$ to be surjective) the invariant set $Y$ equals the entire phase space $X$.
This may be proven in similar fashion to Proposition~\ref{prop:whole phase space} below.
\end{Rem}

Let us explain the usefulness of this Theorem. Quite often, one may
use skew products $\F$ to exhibit various dynamical or ergodic
phenomena (see \cite{GI99}, \cite{GI00}, \cite{GIKN}, \cite{IKS08},
\cite{DG}). One would like to prove the same properties for small
perturbations $\G$ of $\F$, but $\G$ is a priori not a skew product
anymore. However, letting $G=H\circ \G|_Y\circ H^{-1}$, statement a)
of the above theorem implies that $G:\Lambda \times M \rightarrow
\Lambda \times M$ is indeed a skew product: $$ G(b,m)=(h(b),
g_b(m)). $$ One can then study the dynamical properties of the more
mysterious map $\G|_Y$ by studying the dynamical properties of its
conjugate skew product $G$. \\

The fiber maps $g_b$ of the skew product $G$ are $C^1$-close to
those of the skew product $\F|_Y$, in the following sense:
\begin{equation} \label{eqn:close} d(g_b^{\pm 1},f_b^{\pm 1})_{C^1}
\leq O(\rho). \end{equation}

But what can be said about the fiber maps $g_b$ for different $b$'s? Since $\F$ is a $C^1-$morphism,
the fiber maps $f_b$ depend in a $C^1$ fashion on
the point $b\in \Lambda$. Such a result fails for the fiber maps $g_b$,
but statement b) of Theorem~\ref{thm:mt} implies that the
fiber maps $g_b$ depend \ho continuously on the point $b\in \Lambda$:
\begin{equation}
\label{eqn:holfibermaps}
d(g_b^{\pm 1},g_{b'}^{\pm 1})_{C^0} \leq O(d(b,b')^\a),
\end{equation}
where $\a$ is given by \eqref{eqn:defalpha}. A skew product $G$ whose fiber
maps satisfy \eqref{eqn:holfibermaps} will be called a \emph{\ho
skew product}. Thus Theorem~\ref{thm:mt} can be summarized as
follows:

\begin{center} \textbf{Let $\G$ be any small perturbation of a
$C^2$ skew product $\F$ over a $(\l_-, \l, \mu_-, \mu)-$ hyperbolic
map $h$, satisfying the modified dominated splitting
condition \eqref{eqn:dom}. Then $\G$ has an invariant set $Y$ such that the restriction of
$\G|Y$ is conjugated to a \ho skew product
close to $\F|\L \times M$, in the sense of \eqref{eqn:defalpha},
\eqref{eqn:close} and \eqref{eqn:holfibermaps}.} \end{center}

\subsection{The solenoid case}
\label{sub:st}

In this section we will present a partial improvement of
Theorem~\ref{thm:mt} that is inspired by the example of the
Smale-Williams solenoid. Let us begin by introducing and describing
the solenoid map. Fix constants $R\geq 2$ and $\lambda<0.1$, whose
particular values will not be important. Let $B$ denote the solid
torus $$ B = S^1 \times D, \textrm{ where }\ S^1 = \{y \in \rr
/\zz\}, \ D = \{z \in \mathbb C| |z| \le R \}. $$ The \emph{solenoid
map} is defined  as \begin{equation} \label{eqn:sol} h:B \rightarrow
B, \textrm{ }\textrm{ }\textrm{ } h(y, z) = (2y, e^{2\pi i y} +
\lambda z). \end{equation} The maximal attractor of the solenoid
map: $$ \Lambda = \bigcap_{k=0}^\infty h^k(B) $$ is called the
\emph{Smale-Williams solenoid}. It is a hyperbolic invariant set
with contraction coefficient $\lambda$ and expansion coefficient
$\mu^{-1} = 2$ (we take the sup norm in $T_bB$ in the
coordinates $y,z$). Moreover, the estimates in
\eqref{eqn:hypstructure} hold with $\lambda=\lambda_-$ and
$\mu=\mu_-$. \\

We can generalize the above to the following setup: let $B=Z\times
F$ be the product of two compact Riemannian manifolds; $F$ and $B$
may be \mans with boundary. We suppose that $h:B \rightarrow B$ is a
\skpr itself,  i.e. there exists a $C^2-$morphism
$\zeta:Z\rightarrow Z$ such that the following diagram commutes:
\begin{equation} \label{eqn:CD} \begin{CD} B @>{h}>> B \\ @V{\pi}VV
@V{\pi}VV \\ Z @>{\zeta}>> Z , \end{CD} \end{equation} where $\pi$
is the standard projection. We assume that the map $\zeta$
downstairs is expanding, and that the fibers $\{z\}\times F$ are the
stable manifolds of $h$: \begin{equation} \label{eqn:zetaexp}
\mu_-\leq |d\zeta^{-1}|\leq \mu, \end{equation} $$ \lambda_-\leq
|dh|\leq \lambda \textrm{ }\textrm{ }\textrm{ on } \textrm{
}\textrm{ } T(\{z\} \times F) \textrm{ ,}\forall z\in Z. $$ Again,
for technical reasons we assume that $E^s=TF$ is trivialized as in
\eqref{eqn:triv}. In this setup, Theorem~\ref{thm:mt} can be
partially improved by the following result.

\begin{Thm} \label{thm:st} Consider a skew product $\F $ as above
that also satisfies the modified dominated splitting condition.
Then for small enough $\rho>0$, any $\rho-$perturbation $\G $ of
$\F$ has the following properties: \\

a) There exists a continuous map $q : X \to Z$ such that the diagram
\begin{equation}
\label{eqn:semisol}
\begin{CD}
X @>{\G}>> X \\
@V{q}VV @V{q}VV \\
Z @>{\zeta}>> Z
\end{CD}
\end{equation}
commutes. Moreover, the commutative diagrams \eqref{eqn:semi} and
\eqref{eqn:semisol} must be compatible, in the sense that
$q|_Y=\pi \circ p$.  \\

b)The fibers $W^s_z=q^{-1}(z)$ are Lipschitz close to the vertical (constant)
fibers, and \ho continuous in $z$. This means that $W^s_z$ is the
graph of a Lipschitz map $\beta^s_z:F\times M \rightarrow Z$ such that
\begin{equation}
\label{eqn:smallcs}
d(\beta^s_z,z)_{C^0}\leq O(\rho), \textrm{ }\textrm{ }\textrm{ }\mLip \beta^s_z \leq O(\rho)
\end{equation}

\begin{equation}    \label{eqn:holdcs}
d(\beta^s_z ,\beta^s_{z'})_{C^0} \le \frac {d(z,z')^{\a-O(\rho)}}{O(\rho)^\alpha},
\end{equation}
where $\a=\dsp \frac {\ln \mu}{\ln \mu_-}$.
\end{Thm}

As was mentioned above, a particularly important case in which the
theorem applies is the Smale-Williams solenoid with $Z=S^1$, $F=D$ and $h$ given by \eqref{eqn:sol}.

\subsection{Fubini revisited} \label{sub:fubini}

Let $\Sigma^2_+$ be the set of all sequences $\omega^+ =
\omega_0\omega_1\omega_2 \dots $ of zeroes and ones, infinite to the
right with the $(\frac 1 2, \frac 1 2)$ Bernoulli measure. It is
known that for almost all such sequences $\omega^+$, any finite word
$w$ of any length $n$ is encountered within $\omega^+$ with
frequency exactly equal to $2^{-n}$. We are interested in sequences
for which this property fails. More precisely, given $\kappa>0$ and
$w$ a finite word of length $n$, we say that $\om^+$ is
$\kappa,w-$\emph{atypical} if the sequence \begin{equation}
\label{eqn:freq} a_k(\omega^+,w):=\frac {\textrm{number of
occurences of }w\textrm{ among first }k\textrm{ digits of
}\omega^+}{k} \end{equation} has a limit point \emph{outside}
$[2^{-n}-\kappa, 2^{-n}+\kappa]$. The sequence $\omega$ defines a
point $y=\overline{0,\om^+}\in S^1$ written in base 2. We say that
$y$ is $\kappa,w-$\emph{atypical} if $\om^+$ is $\kappa
,w-$atypical. Let $A_{\kappa,w}\subset S^1$ be the set of atypical
points. By the ergodic theorem, it has Lebesgue measure 0 in $S^1$.

\begin{Thm} \label{thm:fubini} Consider a \skpr $\F $ over the
solenoid map as in Theorem~\ref{thm:st}. Let $\zeta : Z \to Z$ be
the duplication of a circle. For any word $w$ and positive $\ka $
\tes $\rho $ \st the following holds. Let $A_{\ka ,w}$ be the same
as above. Then for any $\rho $ perturbation $\G $ of $\F $ and $q$
as in \eqref{eqn:semisol}, we have: \begin{equation}
\label{eqn:semifub2} \emph{mes }q^{-1}(A_{\kappa,w}) = 0.
\end{equation} In other words, the union of $\kappa,w-$atypical
fibers has Lebesgue measure 0 in $X$. \end{Thm}

An analog of this theorem for perturbations of skew products over
the Anosov \diffeo of a two-torus may be proved basing on a recent
result of P. Saltykov \cite{S09}.

\subsection{Attractors with intermingled basins}\label{sub:kan}

The tools developed in this paper allow us to get a new proof of the
following phenomenon discovered by I. Kan:

\begin{Thm}  \label{thm:kan} [\cite{IKS08}, \cite{KS*}] The set of
maps of an annulus $S^1\times [0,1]$ that have intermingled
attracting basins is open in the set of all maps of the annulus into
itself that keep the boundary invariant. \end{Thm}

\emph{Intermingled attracting basins} means the following thing: the
Milnor attractor of the maps mentioned in Theorem \ref{thm:kan}
consists of the two boundary circles, each one having an attracting
basin which is dense and of positive Lebesgue measure. \\

In \cite{IKS08}, \cite{KS*} the theorem above is improved by: \\

\emph{The complement to the union of the attracting basins in the
perturbed  Kan example has Hausdorff dimension smaller than 2.}\\

Theorem~\ref{thm:kan}, in a slightly different form, was claimed in
\cite{K94}; as far as we know, the first proof was obtained in
\cite{Bo}. The same tools also allow us to construct diffeomorphisms
with intermingled attracting basins \cite{I08}; the phase space in
this case is the product of a solid torus and a circle.

\section{Rate of contraction of the graph transform map}
\label{sec:grtrans}

In this section we prove statement $a)$ of Theorem \ref{thm:mt}, and
establish the rate of contraction of the graph transform map, see
Lemma~\ref{lem:1} below. There are two ways to prove statement $a)$.
The first one is to establish partial hyperbolicity
of the \skpr $\F $, and refer to the \hps theory. This theory
implies the semiconjugacy statement $a)$, but gives no estimate of
the rate of contraction of the graph transform map. The second way
is to revisit the graph transform map and to prove simultaneously
the fixed point theorem and the rate of contraction estimate for
this map. This is done in the present section.

\subsection{Laminations}
\label{sub:lam}

Let $B,h,\Lambda$ be as in Subsection~\ref{sub:perhol}. In the
fibers of the bundles $E^s$ and $E^u$ we have the abstract
Riemannian metric, while in the fibers of the trivial bundles $B
\times \rr^k$ and $B \times \rr^l$ we have the standard Euclidean
metric. The isomorphism $\varphi^s$ of \eqref{eqn:triv} implies that
there exist $k$ linearly independent sections of $E^s$. By applying
Gram-Schimdt orthonormalization to these sections, it follows that
there exist $k$ orthonormal sections of $E^s$. Sending a fixed
orthonormal basis of $\rr^k$ to these orthonormal sections will give
us a \emph{metric-preserving} isomorphism $B \times \rr^k
\rightarrow E^s$, and it is this isomorphism that we will henceforth
denote by $\varphi^s$. The same discussion applies to $\varphi^u$.
\\

For any $\delta>0$, we define $Q^s(\delta)$ and $Q^u(\delta)$ to be the open balls of
radius $\delta$ around the origin of $\rr^k$ and $\rr^l$,
respectively. The metric-preserving isomorphisms $\varphi^s$ and
$\varphi^u$ induce metric-preserving isomorphisms in each fiber:
\begin{equation} \label{eqn:trivfiber} \varphi^s_b(\delta):Q^s(\delta) \rightarrow
Q^s_b(\delta), \textrm{ }\textrm{ }\textrm{ } \varphi^u_b(\delta):Q^u(\delta) \rightarrow
Q^u_b(\delta), \end{equation} where $Q^s_b(\delta) \subset E^s$ and $Q^u_b(\delta) \subset
E^u$ are the open balls of radius $\delta$ around the origin in the respective fibers. \\

The number $\delta$ must be chosen small enough such that for any $b\in B$, the
exponential map gives us an open embedding $Q^s_b(\delta) \times Q^u_b(\delta)
\hookrightarrow B$. We write $B_b(\delta)$ for the image of this map.
Composing this embedding with the isomorphism $\varphi^s_b(\delta) \times
\varphi^u_b(\delta)$ gives us an open embedding (coordinate chart):

\begin{equation}
\label{eqn:fixiso} \varphi_b(\delta):Q^s(\delta) \times Q^u(\delta) \hookrightarrow B.
\end{equation}
Let us take $C>\textrm{max}(\lambda_-^{-1},\mu_-^{-1})$, and consider the above constructions for radius $C\delta$.
Then we can express the map $h:B\rightarrow B$ locally around $b$ in the domain and around $h(b)$ in the target.
Therefore, in coordinates given by the chart \eqref{eqn:fixiso}, the map $h$ has the form:

\begin{equation} \label{eqn:hcoord} h_b(\delta) =
(\ph_{h(b)}(C\delta))^{-1} \circ h \circ \ph_b(\delta), \textrm{
}\textrm{ }\textrm{ } {(h^{-1})}_b(\delta) = (\ph_{b}(C\delta))^{-1}
\circ h^{-1} \circ \ph_{h(b)}(\delta). \end{equation}

For various values of $\delta$, the maps $h_b(\delta)$ will
represent the same germ at $0$, but will have different domains.
Similarly, the maps ${(h_b)}^{-1}(\delta)$ and
${(h^{-1})}_b(\delta)$ are representatives of the same germ at $0$,
but have different domains. \\

Until Section~\ref{sec:hold}, we will work with a single, fixed
$\delta$. Therefore, we will often write simply
$Q^s,Q^u,Q^s_b,Q^u_b,B_b,\ph^s_b,\ph^u_b,\ph_b,h_b,(h^{-1})_b$ for
the notions introduced in the previous paragraphs. By
\eqref{eqn:hypstructure} and the fact that \diffeos
\eqref{eqn:trivfiber} are metric-preserving, $dh_b$ has
block-diagonal form at $0$: $$ dh_b(0) = \textrm{diag} (A^u, A^s),
$$ where $\lambda_-\leq |A_s|\leq \lambda$ and $\mu_-\leq
|A_u^{-1}|\leq \mu$. Because the coordinate charts $\ph_b$ are
smooth functions, we have the following estimate throughout
$Q^s\times Q^u$: \begin{equation} \label{eqn:esth} {||dh_b -
\textrm{diag} (A^u, A^s)||}_{C^0} \le O(\delta ). \end{equation}

Now consider another compact manifold $M$, as in the statement of Theorem~\ref{thm:mt}.
For any domain $A$ and any mapping $\b : A \to B$, we will denote by $\ga (\b )$ the map from $A$ onto the graph:

\begin{equation}  \label{eqn:gamma} \ga(\b):A \rightarrow A \times
B, \textrm{ }\textrm{ }\textrm{ } \ga (\b ): a \mapsto (a, \b (a))
\in A \times B.\end{equation}  Statement $a)$ of Theorem
\ref{thm:mt} provides a correspondence between leaves and base
points, so it's about time we defined these. The leaves of
center-stable, center-unstable and center foliations corresponding
to $b \in \L$ are represented by Lipschitz maps: \begin{equation}
\label{eqn:leaf} \beta^s_b:Q^s\times M \rightarrow Q^u,\textrm{
}\textrm{ } \textrm{ }\beta^u_b:Q^u\times M \rightarrow Q^s,\textrm{
}\textrm{ } \textrm{ }\beta_b:M \rightarrow Q^u \times Q^s.
\end{equation} Then we define the leaves to be simply the graphs of
the Lipschitz maps, embedded in $B\times M$ via \eqref{eqn:fixiso}:

\begin{equation} \label{eqn:leaf1} W_b^* = \mbox{Im}(\ph_b \times
\textrm{Id}) \circ \ga (\b_b^*) \end{equation}
Here and below, $*$ stands for $s, u$ or blank space. \\

Intuitively, $W_b^s$ denotes a center-stable leaf, $W_b^u$ denotes a central-unstable leaf,
while $W_b$ denotes a central leaf. We will never consider
strongly stable or unstable leaves. \\

We now define certain functional spaces $\B^*$ of maps $\b^*$. These are, by definition, the
spaces of Lipschitz  maps \eqref{eqn:leaf} that satisfy the condition:

\begin{equation} \label{eqn:max} \max \left\{   \| \b^* \|_{C^0}, \frac {\mLip \b^*}D \right\}
\le \frac {\delta}2 \end{equation}
Here, $D$ is a constant that will be picked in the proof of Lemma~\ref{lem:1}.
The norm on the spaces $\B^*$ will always be the $C^0$ norm, and will be denoted by $||\cdot||$. \\

Intuitively speaking, a central-stable, central-unstable or central
\emph{lamination} is a continuous assignment of leaves as $b$ runs
over $\Lambda$. Rigorously speaking, a lamination is a continuous map:

\begin{equation} \label{eqn:section} S^*:\Lambda \rightarrow \BB^*.
\end{equation}
The map $S^*$ is completely determined by the continuous collection of maps $\b^*_b=S^*(b)$, as $b$ ranges over $\Lambda$.
Equivalently, $S^*$ is completely determined by the leaves $W^*_b$ of these maps. \\

The space of continuous sections $S^*$ as above is denoted
by $\Ga^*$. The norm in this space is again the $C^0$ norm:

$$ \| S^*\| = \max_{b \in \L} \|S^*(b)\| .$$
For any $\delta > 0$ small enough, the metric space $\Ga^*$ with the distance $\rho (S^*_1, S^*_2) = ||S^*_1 - S^*_2||$ is complete.
Indeed, if $\b^*_n \to \b^*$ and $\mLip \b^*_n \le D\delta/2 $, then
$\mLip \b^* \le D\delta/2$. \\

Now consider a map $\G:B\times M \rightarrow B \times M$, like in the setup of Theorem~\ref{thm:mt}.
A central-stable, central-unstable or central lamination is called
$\G-$invariant if its leaves $W^*_b$ satisfy: \begin{equation}
\label{eqn:invfol1} \G(W^{s}_{b}) \subset W^s_{h(b)}, \textrm{
}\textrm{ }\textrm{ } W^u_{h(b)} \subset \G(W^u_b) \end{equation}
\begin{equation} \label{eqn:invfol2}\textrm{or }\textrm{ }\textrm{ } W_{h(b)} = \G(W_b).
\end{equation} These conditions can all be written in terms of the
maps $\beta^*_b$ defining these leaves, and thus in terms of
laminations $S^*$ themselves. This will be done in the beginning
of Subsection~\ref{sub:graph}. \\

Our plan for the proof of Statement $a)$ of Theorem~\ref{thm:mt} will
be the following: we will use the graph transform method described
in the following subsection to find $\G-$invariant central-stable
and central-unstable laminations. Then the central lamination will be
given by $$ W_b=W^s_b\cap W^u_b. $$ Property \eqref{eqn:invfol2}
will follow from \eqref{eqn:invfol1}, so the central lamination will
also be invariant under $\G$. Once we have the central lamination,
we will define $$ Y=\bigsqcup_{b\in \Lambda} W_b. $$ Sending $W_b$
to $b$ defines the desired projection map $p:Y\rightarrow \Lambda$
of \eqref{eqn:semi}. Then the $\G-$invariance of the central
lamination is precisely equivalent to the commutativity of diagram
\eqref{eqn:semi}. We will follow this plan in the next subsections.

\subsection{The graph transform map}
\label{sub:graph}

Here we will deal with the $*=s$ case only, since the $*=u$ case is treated similarly.
After that, the central case will be treated as described above.
We will introduce first a ``pointwise'' graph transform map:

$$
\g_b:\B^s \longrightarrow \B^s
$$
that acts on single leaves, and then a ``global'' graph transform map:

$$ \g:\Ga^s \longrightarrow \Ga^s $$ that acts on entire
laminations. In both cases, the geometric idea is the same: start
with a map $\beta^s:Q^s \times M \longrightarrow Q^u$ as in
\eqref{eqn:leaf}. Take the corresponding leaf $W_{h(b)}^s \subset B
\times M$, and take its inverse image under $\G$. The claim is that
we obtain a different leaf $\overline{W}_{b}^s \subset B \times M$,
corresponding to a map $\overline{\beta}^s: Q^s \times M
\longrightarrow Q^u$. Then we define the graph transform map as: $$
\g_b(\beta^s)=\overline{\beta}^s. $$ In other words, the graph
transform is implicitly defined by the following relation:

\begin{equation} \label{eqn:implicitdef}
\{\G^{-1}(\ph_{h(b)}(x_s,\beta^s(x_s,m)),m)\}=\{(\ph_b(x_s,\overline{\beta}^s(x_s,m)),m)\}.
\end{equation} We will prove in the appendix that the above
correctly defines $\overline{\beta}^s$ (in other words, that the
Implicit Function Theorem applies). The above definition also works
in families. For a lamination $S^s \in \Gamma^s$ with leaves that
are graphs of $\beta^s=S^s(h(b))$, define its graph transform as: $$
\g(S^s)=(\overline{S}^s), $$ where
$\overline{S}^s(b)=\overline{\beta}^s$ is defined by relation
\eqref{eqn:implicitdef}. \\

Comparing with \eqref{eqn:invfol1}, we see that a lamination $S^s$ is $\G-$invariant
if and only if it is a fixed point of the graph transform map $\g$.
Therefore, to show that there exists a unique $\G-$invariant central-stable lamination,
we will use the fixed point principle: it is enough to show that $\g$ is well defined and contracting.

\begin{Lem}
\label{lem:1}
For $\rho$ small enough and any $\F, \G$
as in Theorem~\ref{thm:mt}, there exists $\delta=O(\rho)$ so
that the graph transform $\g$ maps $\Ga^s$ into itself and is contracting
with Lipschitz constant $\mu+O(\delta)$. In other words, for any $S^s_0,
S^s_1 \in \Ga^s $ we have: \begin{equation} \label{eqn:cont1} \|
\g(S_0^s) - \g(S_1^s) \| \le (\mu + O(\delta))\| S_0^s - S_1^s\| .
\end{equation}
In the pointwise situation, for any $b\in \Lambda$, we claim that $\g_b$ maps $\B^s$ into itself.
Furthermore, for any $\b^s_0,\beta^s_1 \in \B^s$, we have:
\begin{equation}  \label{eqn:36a}
\| \g_b(\b_0^s) - \g_b(\b^s_1)\| \leq (\mu + O(\delta ))\|  {\b^s_0} - {\b^s_1} \|.
\end{equation}
\end{Lem}

\begin{Cor} \label{cor:fixed}
For $\delta=O(\rho)$ small enough, the graph transform map $\g$ has a unique fixed point in $\Ga^s$. \end{Cor}

\begin{proof} The statements about the global graph transform immediately follow from the
corresponding statements in the pointwise case. So let us start by proving that $\g_b$ maps $\B^s$ to itself.
Take $b\in \La$, $\b^s \in \B^s$ and let $\overline{\beta}^s = \g_b(\b^s)$. We need to prove that:
\begin{equation} \label{eqn:cnorm} \| \overline{\b}^s \| \le \frac {\delta}2,
 \end{equation}
\begin{equation} \label{eqn:lip} \mLip \overline{\b}^s \le \frac
{D\delta}2. \end{equation} Recall that $\gamma_\b $ is the map of
$Q^s \times M$ onto the graph of $\b^s$, see \eqref{eqn:gamma}. In
the Appendix we prove that for any $\b = \b^s$ that satisfies
\eqref{eqn:max}, there exists a Lipshitz homeomorphism $G_{\bar \b ,
b}: Q^s \times M \to Q^s \times M$, see \eqref{eqn:Gbb},  such that
$$ \overline{\b}^s = \pi_u \circ \G_b^{-1} \circ \gamma (\b^s) \circ
G_{\overline\b, b}.$$ Here $\G_b ={(\ph_{h(b)}\times
\textrm{Id})}^{-1} \circ \G \circ (\ph_b \times \textrm{Id})$.  Note
that $$ {||\overline{\b}^s||} \le ||\pi_u \circ \G_b^{-1} \circ
\gamma(\b^s)||,$$ because the shift in the argument of the right
hand side does not change the $C^0$ norm. Therefore, by
\eqref{eqn:perturbation}, we have: $$||{\overline{\b}}^s|| \le
||\pi_u \circ \F_b^{-1} \circ \gamma(\b^s)|| + O(\rho)  = ||\pi_u
\circ (h^{-1})_b \circ \gamma(\b^s)|| + O(\rho). $$ By
\eqref{eqn:esth}, we can further estimate the above: $$
||{\overline{\b}}^s|| \leq (\mu + O(\delta))||\b^s||+O(\rho). $$
Since $\mu<1$ and $||\b^s||\leq \delta/2$, for appropriately chosen
$\rho=O(\delta)$ the above can be made $\leq \delta/2$. This proves
\eqref{eqn:cnorm}. As for \eqref{eqn:lip}, note that
\begin{equation} \label{eqn:estlip} \mLip {\overline{\b}^s} \le
\mLip (\pi_u \circ \G_b^{-1}\circ \gamma(\b^s)) \cdot \mLip
G_{\overline\b, b}. \end{equation} We need to show that the right
hand side of the above is $\leq D\delta/2$. It is enough to do this
for $\b^s$ and $\overline\beta^s$ of class $C^1$, since these maps
are dense in $\B^s$. In this $C^1$ case, we have:

$$
d(\pi_u \circ \G_b^{-1} \circ \gamma(\b^s)) =  d(\pi_u \circ \G_b^{-1})\circ \gamma(\b^s) \cdot d\gamma(\b^s) \leq
$$

$$
\leq \left[d(\pi_u \circ \F_b^{-1})\circ \gamma(\b^s) +O(\rho)\right]\cdot d\gamma(\b^s) \leq
$$

$$ \leq \left[\left( \begin{array}{ccc} \frac {\partial \pi_u \circ
h^{-1}_b}{\partial x_s} & \frac {\pi_u \circ \partial
h^{-1}_b}{\partial x_u} & 0 \\ \end{array} \right) \circ
\gamma(\b^s) +O(\rho) \right] \cdot \left( \begin{array}{cc} 1 & 0
\\ \frac {\partial \beta^s}{\partial x_s} & \frac {\partial
\beta^s}{\partial m} \\ 0 & 1 \\ \end{array} \right) \leq $$

$$ \leq \left[\left( \begin{array}{ccc} 0 & \mu & 0 \\ \end{array}
\right)+O(\delta)+O(\rho) \right] \cdot \left( \begin{array}{cc} 1 &
0 \\ \frac {\partial \beta^s}{\partial x_s} & \frac {\partial
\beta^s}{\partial m} \\ 0 & 1 \\ \end{array} \right) \leq $$

\begin{equation} \label{eqn:ppp} \leq \left( \begin{array}{cc}
\mu\cdot \frac {\partial \beta^s}{\partial x_s} & \mu\cdot \frac
{\partial \beta^s}{\partial m}  \\ \end{array} \right)+O(\delta)+
O(\rho)\leq \mu\cdot  \mLip \beta^s + O(\delta), \end{equation}
$$$$
since $\rho=O(\delta)$. Combining this estimate with
Proposition~\ref{prop:composition} of the Appendix, we see that: $$
\mLip {\overline\beta^s} \leq (\mu \cdot \mLip \beta^s +
O(\delta))\cdot (L + O(\delta))\cdot (1+\textrm{Lip }\overline\b^s).
$$ Since $\mLip \beta^s \leq D\delta/2$, the above gives us: $$
\mLip {\overline\beta^s} \leq \frac {\mu L \cdot D\delta/2+L\cdot
O(\delta)+O(\delta^2)}{1-\mu L \cdot D\delta/2-L\cdot
O(\delta)-O(\delta^2)}. $$ By assumption \eqref{eqn:dom}, we have
$\mu L<1$. Therefore, if we pick the constant $D$ large enough (but
still requiring that $D\delta<<1$), the right hand side of the above
will be $\leq D\delta/2$. This proves \eqref{eqn:lip}. \\

Now that we have proved $\g$ and $\g_b$ to be well-defined, let us
pass to proving \eqref{eqn:cont1} and \eqref{eqn:36a}. As we said
before, the second inequality implies the first, so we will only
prove the second one. As above, write
$\overline\b^s_0=\g_b(\beta^s_0)$ and
$\overline\b^s_1=\g_b(\beta^s_1)$. From \eqref{eqn:graph}, we see
that:

\begin{equation}
\label{eqn:estimates0}
||\overline\b^s_0-\overline\b^s_1|| \leq T_1+T_2,
\end{equation}
where:

$$
T_1 =||\pi_u \circ \G_{b}^{-1} \circ \ga({\b^s_0})\circ G_{\overline\b_0,b} -\pi_u \circ \G_{b}^{-1}
\circ \ga({\b^s_1})\circ G_{\overline\b_0,b}||,
$$

$$
T_2=||\pi_u \circ \G_{b}^{-1} \circ \ga({\b^s_1}) \circ G_{\overline\b_0,b} - \pi_u \circ \G_{b}^{-1}
\circ \ga({\b^s_1}) \circ G_{\overline\b_1,b}||.
$$
As it will soon be clear, $T_1$ is the dominant term:

$$ T_1 \leq \textrm{Lip }(\pi_u \circ \G_{b}^{-1}) \cdot
||\ga({\b^s_0}) -\ga({\b^s_1})||.$$ The second factor in the right
hand side is $\leq ||{\b^s_0} -{\b^s_1}||$. As for the first factor,
we see that: $$ \textrm{Lip }(\pi_u \circ \G_{b}^{-1}) \leq
\textrm{Lip }(\pi_u \circ \F_{b}^{-1}) + O(\rho)=\textrm{Lip }(\pi_u
\circ h_{b}^{-1}) + O(\rho) \leq \mu+O(\rho).$$ Since
$\rho=O(\delta)$, we conclude that:

\begin{equation}
\label{eqn:t1}
T_1 \leq (\mu+O(\delta))\cdot ||{\b^s_0} -{\b^s_1}||.
\end{equation}
As for $T_2$, we see that:

$$ T_2\leq \textrm{Lip }(\pi_u \circ \G_{b}^{-1} \circ
\ga({\b^s_1}))\cdot ||G_{\overline\b_0,b} - G_{\overline\b_1,b}||.
$$ In \eqref{eqn:ppp}, we saw that: $$ \mLip (\pi_u \circ
\G_{b}^{-1} \circ \gamma({\b^s_1})) \leq O(\delta). $$ In
Proposition~\ref{prop:last} of the Appendix, we will prove that $$
||G_{\overline\b_0,b} - G_{\overline\b_1,b}|| \leq O(1)\cdot
||\overline\beta^s_0-\overline\beta^s_1||.$$ Therefore, we obtain:

$$
T_2 \leq O(\delta)\cdot ||\overline\beta^s_0-\overline\beta^s_1||.
$$
Together with \eqref{eqn:t1}, this implies:
$$
||\overline\b^s_0-\overline\b^s_1|| \leq (\mu+O(\delta))\cdot ||{\b^s_0} -{\b^s_1}|| + O(\delta)
\cdot ||\overline\beta^s_0-\overline\beta^s_1|| \Rightarrow
$$
$$
||\overline\b^s_0-\overline\b^s_1|| \leq (\mu+O(\delta))\cdot ||{\b^s_0} -{\b^s_1}||.
$$
This is precisely the desired inequality \eqref{eqn:36a}.
\end{proof}

\subsection{The central lamination}
\label{sub:central}

Corollary \ref{cor:fixed} tells us that there exists a unique
$\G-$invariant central stable lamination $S^s \in \Ga^s$. This can
be presented either via the maps $\beta^s_b$, or via the leaves
$W^s_b$ (as $b$ ranges over $\Lambda$). Similarly, there exists a
unique $\G-$invariant central unstable lamination $S^u\in \Ga^u$.
Let us define the central lamination $S$ via its leaves $W_b$, which
we define by: \begin{equation} \label{eqn:defcentral}
W_b=W^s_b\cap W^u_b. \end{equation} This lamination will be
$\G-$invariant, in the sense of \eqref{eqn:invfol2}. Let us describe
$W_b$ more explicitly. By definition, $$ W^s_b=\textrm{Im
}(\varphi_b \times \textrm{Id})\{(x_s,\beta^s_b(x_s,m),m)|\textrm{
}x_s\in Q^s, m\in M\} $$ $$ W^u_b=\textrm{Im }(\varphi_b \times
\textrm{Id})\{(\beta^u_b(x_u,m),x_u,m)|\textrm{ }x_u\in Q^u, m\in
M\}, $$ where $\beta_b^s,\beta_b^u$ have Lipschitz norms at most
$D\delta/2<<1$. Then, for each $m\in M$, the system of equations
\begin{equation} \label{eqn:system} \left\{
  \begin{array}{ll}
  x_s=\beta^u_b(x_u,m) \\
  x_u=\beta^s_b(x_s,m)
  \end{array}
\right. \end{equation} has a unique solution
$(x_s,x_u)=:\beta_b(m)\in Q^s\times Q^u$. Indeed, for any fixed $m$
the maps $\b^s_b \circ \b^u_b: Q^u \to Q^u$ and $\b^u_b \circ
\b^s_b: Q^s \to Q^s$ are Lipshitz with constant $\le {(L\delta
)}^2 << 1$. Then each of the two maps is contracting and has a
unique fixed point: call these $x_u$ and $x_s$, respectively. Then the pair $(x_s,x_u)$ is the solution of \eqref{eqn:system}, and the above map $\beta_b$ is well-defined. If we define the map
$\widetilde{\beta}_b=\varphi_b(\beta_b):M \rightarrow B$, then its
graph is precisely $W_b$: $$
W_b=\{(\widetilde{\beta}_b(m),m)|\textrm{ }m\in M\}. $$ Because it
is the intersection of an invariant central-stable lamination with
an invariant central-unstable lamination, $S=(\beta_b)=(W_b)$ is an
invariant central lamination. \\

It is not hard to see from \eqref{eqn:system} that
\begin{equation}
\label{eqn:smallb}
||\beta_b||_{C^0} \leq \frac {\delta}2 \textrm{ }\textrm{ }\textrm{ and }\textrm{ }\textrm{ }\frac {\mLip \beta_b}D \leq \delta.
\end{equation}
Because the chart $\ph_b$ is metric preserving at 0 and smooth in the domain $Q^s\times Q^u$ (which has diameter of order $\delta$),
we have:
\begin{equation}
\label{eqn:smalltildeb}
d(\widetilde{\beta}_b,b)_{C^0} \leq O(\delta)=O(\rho) \textrm{, }\textrm{ }\textrm{ }\textrm{ }\textrm{ }\mLip
\widetilde{\beta}_b \leq O(\delta)=O(\rho).
\end{equation}

\begin{proof} \textbf{of statement a) of Theorem~\ref{thm:mt}: } Start from the $\G-$invariant central
lamination $S$ constructed above, and define $Y=\bigcup_{b\in \Lambda} W_b$. This union is obviously an invariant set of $\G$,
and moreover the following proposition implies that it is actually a disjoint union.

\begin{Prop} \label{prop:disjointfibers} For all $b\neq b' \in
\Lambda$, the corresponding central leaves are disjoint: $$ W_b\cap
W_{b'}=\emptyset. $$ \end{Prop}

\begin{proof} Let us assume by contraposition that $W_b\cap
W_{b'} \neq \emptyset$. By the $\G-$ invariance of the lamination,
then $$ W_{h^k(b)}\cap W_{h^k(b')} \neq \emptyset, $$ for all $k\in
\zz$. Pick a point $(\tilde{b},m)$ in the above non-empty
intersection. Then \begin{equation} \label{eqn:tildeb}
\widetilde{b}=\widetilde{\beta}_{h^k(b)}(m)=\widetilde{\beta}_{h^k(b')}(m)
\end{equation} By \eqref{eqn:smalltildeb}, the point
$\widetilde{\beta}_{h^k(b)}(m)$ is at distance at most $O(\rho)$
from $h^k(b)$. Similarly, $\widetilde{\beta}_{h^k(b')}(m)$ is at
distance at most $O(\rho)$ from $h^k(b')$. This implies that
$h^k(b)$ and $h^k(b')$ are at most $2\cdot O(\rho)$ apart, for all
$k\in \mathbb{Z}$. This is obviously impossible for $\rho$ small
enough, because for such $\rho $, the quantity $ O(\rho )$  is
smaller than the expansivity constant of $h$. \\ \end{proof}

Therefore, the map $p:Y \rightarrow \Lambda$ sending $W_b$ to $b$ is well-defined. Moreover,
the $\G-$invariance of the lamination $S=(W_b)$ is precisely equivalent to the commutativity of the diagram \eqref{eqn:semi}.
The continuity of $p$ follows from the continuity of our laminations, and this also
implies that the map $H$ of \eqref{eqn:homeo1} is continuous. \\

Note that the map $H$ is bijective, with inverse given by $H^{-1}(b,m)=(\widetilde{\beta}_b(m),m)$.
The map $H^{-1}$ is clearly continuous in $m$, and continuity in $b$ follows from the \ho continuity statement \eqref{eqn:hol},
which will be proved in the next subsection. Therefore $H$ is a homeomorphism, thus concluding the proof of statement $a)$.
\end{proof}

\section{H\"{o}lder continuity of the central lamination}
\label{sec:hold}

This section will be concerned with the proof of statement $b)$ of
Theorem~\ref{thm:mt}. By definition, we have
$p^{-1}(b)=W_b=\textrm{Graph}(\widetilde{\beta}_b)$, where
$\widetilde{\beta}_b$ satisfies relations \eqref{eqn:smalltildeb}. This
is precisely the requirement \eqref{eqn:fibers}. In this section
we will prove the rest of statement $b)$, which refers to \ho
continuity. \\

First, for any $b\in \Lambda$, we will define its \emph{local
central-stable} and \emph{central-unstable manifolds} as $$ V^s_b=\{b'\in B|d(h^n(b'),h^n(b))\leq \delta
\textrm{, }\forall n\geq 0\}, $$ $$ V^u_b=\{b'\in
B|d(h^{-n}(b'),h^{-n}(b))\leq \delta \textrm{, }\forall n\geq 0\}.
$$

\begin{Prop} \label{prop:inthyp} Let $h, \Lambda $ and $d$ be the
same as at the beginning of Subsection~\ref{sub:perhol}. Then the
following statements hold for all $b,b'\in \Lambda$:
\begin{enumerate}
  \item $\emph{if } b'\in V^s_b \emph{ and } d(h^{-1}(b), h^{-1}(b'))\leq \delta \Rightarrow h^{-1}(b')\in V^s_{h^{-1}(b)}$
  \item $\emph{if } b'\in V^u_b \emph{ and } d(h(b), h(b'))\leq \delta \Rightarrow h(b')\in V^u_{h(b)}$
  \item $\emph{if } b'\in V^s_b \Rightarrow \lambda_- - O(\delta) \leq \dsp \frac {d(h(b),h(b'))}{d(b,b')}\leq \lambda+O(\delta) $
  \item $\emph{if } b'\in V^u_b \Rightarrow \mu_- - O(\delta) \leq \dsp \frac {d(h^{-1}(b),h^{-1}(b'))}{d(b,b')} \leq \mu+O(\delta)$
\end{enumerate}
\end{Prop}

\begin{proof} Statements 1 and 2 follow immediately from the
definitions of $V^s_b,V^u_b$. We will now prove Statements $3$ and
$4$. The map $h$ has invariant stable and unstable laminations;
$V^s_b, V^u_b$ are the leaves of these laminations. They are smooth
manifolds, and $V^s_b (V^u_b)$ is tangent at $b$ to $E^s (E^u)$. Now
Statements $3$ and $4$ follow from \eqref{eqn:hypstructure} and the
$C^2$-smoothness of $h$. \end{proof}

We further ask that  $h$ has the following \emph{local product
structure}: for all $b,b'\in \Lambda$ such that $d(b,b')\leq
\delta$, there exists a unique $b^*\in B$ such that \begin{equation}
 \label{eqn:prodstr} V^u_b \cap V^s_{b'}=\{b^*\},\end{equation} and moreover: \begin{equation}
\label{eqn:prod} d(b,b^*)+d(b',b^*)\leq O(d(b,b')). \end{equation}
This property is easily seen to hold for linear Anosov
diffeomorphisms of the torus, because then $V^s_b$ and $V^u_{b'}$
are just straight lines that meet transversely under a fixed angle
independent of $b,b'$. It also holds for the Smale-Williams
solenoid, because then $V^s_{b}=\{y(b)\}\times D$ and $V^u_{b'}$ is
a curve that intersects $V^s_{b}$ transversely (such that the angle
between $V^s_b$ and $V^u_{b'}$ is separated from zero). \\

\begin{proof} \textbf{of statement b) of Theorem~\ref{thm:mt}: } We
have already proved the closeness property \eqref{eqn:fibers} in
relation \eqref{eqn:smalltildeb} above. As for the H\"older property
\eqref{eqn:hol}, it is enough to prove it for $b,b'$ which are at
most $\delta $ apart. Indeed, for any $\a > 0$ and any $b,b'$ with
$d(b,b') > \delta $, we have by default: \begin{equation}
\label{eqn:hhol} d(h(b), h(b')) \le Cd^\a (b,b'),\end{equation}
where $C = \frac {\mbox{diam} B}{\delta^\a }$. Therefore, we can
restrict attention to $b,b'$ that are such that the unique point
$b^*$ of \eqref{eqn:prodstr} satisfies:

\begin{equation}
\label{eqn:closeness} d(b,b^*)  \leq \delta, \textrm{ }d(b',b^*) \leq \delta, \textrm{ }
d(b,b') \leq \delta, \end{equation} For such nearby $b,b'$, we
essentially need to estimate the distance between the maps
$\widetilde{\beta}_b,\widetilde{\beta}_{b'}:M \rightarrow B$. These
maps were defined by the condition that their graphs coincide with
$W^s_b \cap W^u_b$ and $W^s_{b'}\cap W^u_{b'}$, respectively. \\

However, this is a bit of an issue: \emph{different} leaves $W^s_b$
and $W^s_{b^*}$ (and also their $u$ counterparts) are defined using
\emph{different} coordinate charts $\ph_b$ and $\ph_{b^*}$. To
resolve this problem in the $s$-case (the $u$-case is treated in
the same way), let us write $W^s_{b}$ as the graph of a function
$\overline{\b}_b^s:Q^s \times M \rightarrow Q^u$ in the coordinate
chart $\ph_{b^*}$: $$ \{(\ph_b \times
\textrm{Id}_M)(x_s,\beta^s_b(x_s,m),m)\}
\stackrel{\textrm{def}}{=}W^s_b =\{(\ph_{b^*} \times
\textrm{Id}_M)(x_s,\overline{\beta}^s_b(x_s,m),m)\} $$ The function
$\overline{\beta}^s_b$ is defined uniquely and implicitly by the
above relation, but we must require the inclusion
$\textrm{Im}(\ph_b) \subset \textrm{Im}(\ph_{b^*})$. We certainly
cannot ensure this if we define the charts $\ph_b$ and $\ph_{b^*}$
with respect to the same $\delta$ in \eqref{eqn:fixiso}. But if we
define $\ph_{b^*}$ with respect to $3\delta$ instead of $\delta$
(i.e. define the chart on a neighborhood 3 times bigger), then the
desired inclusion becomes a consequence of \eqref{eqn:closeness}.

\begin{Def} With the assumption \eqref{eqn:closeness},  we define the
distance between the leaves corresponding to $d, b^*$, to be $$
d(W^s_b,W^s_{b^*}) := ||\beta^s_{b^*}-\overline{\beta}^s_b||. $$
\end{Def}

Implicit in the definition is the fact that the right hand side only
makes sense on the domain of $\overline{\beta}^s_b$, which as was
said before, is contained in the domain of $\beta^s_{b^*}$. Note
that the above definition is \emph{not} symmetric in $b$ and $b^*$.
\\

Now we must look at what happens with these leaves under the graph
transform. Take two leaves $W^s_{h(b)}$ and $W^s_{h(b^*)}$, given in
the coordinate chart $\ph_{h(b^*)}$ by maps
$\overline{\beta}^s_{h(b)}$ and $\beta^s_{h(b^*)}$, respectively.
Then take their images under the graph transform $W^s_{b}$ and
$W^s_{b^*}$, given in the coordinate chart $\ph_{b^*}$ by maps
$\overline{\beta}^s_{b}$ and $\beta^s_{b^*}$. The inequality
\eqref{eqn:36a} of Lemma \ref{lem:1} precisely says that:
\begin{equation} \label{eqn:contWs} d(W^s_b, W^s_{b^*}) \leq
(\mu+O(\delta)) \cdot d(W^s_{h(b)}, W^s_{h(b^*)}). \end{equation}
Doing the analogous computations for central-unstable foliations, we
see that: \begin{equation} \label{eqn:contWu} d(W^u_{b'}, W^u_{b^*})
\leq (\lambda+O(\delta)) \cdot d(W^u_{h^{-1}(b')},
W^u_{h^{-1}(b^*)}). \end{equation} Now recall that we fixed points
$b,b'$ satisfying relation \eqref{eqn:closeness}. Let us consider
the positive integers: $$ k=\left\lfloor \log_{\mu_--O(\rho)} \frac
{d(b,b^*)}{\delta}\right\rfloor, \textrm{ }\textrm{ }\textrm{
}l=\left\lfloor \log_{\lambda_--O(\rho)} \frac
{d(b',b^*)}{\delta}\right\rfloor. $$ Iterate relation
\eqref{eqn:contWs} $k$ times, and we obtain: $$ d(W^s_b, W^s_{b^*})
\leq (\mu+O(\delta))^k \cdot d(W^s_{h^k(b)}, W^s_{h^k(b^*)}). $$ By
the definition of $k$ (and property 4 of
Proposition~\ref{prop:inthyp}), $k$ is the biggest positive integer which
would ensure that the points $h^k(b)$ and $h^k(b^*)$ remain at most
distance $\delta$ apart. Indeed, if they were at a bigger distance
apart, the entire discussion above would break down. But since the
distance between $h^k(b)$ and $h^k(b^*)$ is at most $\delta$, we
infer that the distance between the corresponding leaves is also at
most $O(\delta)$. Therefore, the above inequality implies:

$$ d(W^s_b, W^s_{b^*}) \leq (\mu+O(\delta))^k \cdot O(\delta) \leq
d(b,b^*)^{\frac {\ln \mu}{\ln \mu_-}-O(\delta)} \cdot O(1). $$ The
analogous discussion with $l,\lambda,b',u$ instead of $k,\mu,b,s$
gives us: $$ d(W^u_{b'}, W^u_{b^*}) \leq (\lambda+O(\delta))^l \cdot
O (\delta )\leq d(b',b^*)^{\frac {\ln \lambda}{\ln
\lambda-}-O(\delta)} \cdot O(1). $$   Letting $\alpha$ be defined as
in \eqref{eqn:defalpha}, the above relations give us:

\begin{equation} \label{eqn:aa} d(W^s_b, W^s_{b^*}) \leq
d(b,b^*)^{\alpha-O(\delta)} \cdot O(1),\ \ \textrm{ }d(W^u_{b'},
W^u_{b^*}) \leq d(b',b^*)^{\alpha-O(\delta)} \cdot O(1).
\end{equation}
Let's now prove that \begin{equation} \label{eqn:inc} W^s_{b'}
\subset W^s_{b^*}, \textrm{ }\textrm{ }\textrm{ and analogously
}\textrm{ }\textrm{ } W^u_b \subset W^u_{b^*}. \end{equation}
Relation \eqref{eqn:contWs} for $b$ replaced with $b'$ becomes: $$
d(W^s_{b'}, W^s_{b^*}) \leq (\mu+O(\delta)) \cdot d(W^s_{h(b')},
W^s_{h(b^*)}). $$ However, since $b'\in V^s_{b^*}$, then the map $h$
actually brings the points $b'$ and $b^*$ closer together (by
property 3 of Proposition~\ref{prop:inthyp}). So we can iterate the
above inequalities as many times as we want. We see that: $$
d(W^s_{b'}, W^s_{b^*}) \leq (\mu+O(\delta))^i \cdot d(W^s_{h^i(b')},
W^s_{h^i(b^*)}), $$ for any $i>0$. As $i \rightarrow \infty$, this
implies $d(W^s_{b'}, W^s_{b^*})=0$.  This proves \eqref{eqn:inc} in
the $s$-case. The proof in the $u$-case is similar. \\

We can now turn to the proof of \eqref{eqn:hol}, thus completing the proof of
Theorem~\ref{thm:mt}. Recall that for any $b \in B, \ W_b =
W^s_b \cap W^u_b$. Let us first prove that \begin{equation}
\label{eqn:star} d(W_b, W_{b^*}) \le d(b, b_*)^{\a -O(\delta )}\cdot O(1),
\ d(W_{b'}, W_{b^*}) \le d(b', b_*)^{\a -O(\delta )}\cdot O(1).
\end{equation}
By \eqref{eqn:inc} we have: \begin{equation}  \label{eqn:center} W_b
= W^s_b \cap W^u_{b}, \ W_{b^*} = W^s_{b^*} \cap W^u_{b^*}.
\end{equation} In the chart $\ph_{b^*} \times \textrm{Id}$, the leaves
$W^s_{b^*}$, $W^u_{b^*}$, $W^s_b$, $W^u_{b}$, $W^s_{b'}$, $W^u_{b'}$
are given by maps $\beta^s_{b^*}$, $\beta^u_{b^*}$,
$\overline{\beta}^s_b$, $\overline{\beta}^u_{b}$,
$\overline{\beta}^s_{b'}$, $\overline{\beta}^u_{b'}$. Then
\eqref{eqn:aa} gives us: $$ ||\beta^s_{b^*}-\overline{\beta}^s_b||,
\textrm{ }||\beta^u_{b^*}-\overline{\beta}^u_{b'}|| \leq
d(b,b')^{\alpha-O(\delta)} \cdot O(1), $$ while \eqref{eqn:inc}
gives us: $$ \beta^s_{b^*} = \overline{\beta}^s_{b'}, \textrm{ }
\beta^u_{b^*}=\overline{\beta}^u_b. $$ Of course, when one reads the
above inequalities, one should keep in mind that the maps
$\beta^{s,u}_{b^*}$ are defined on a neighborhood 3 times bigger
than the maps $\overline{\beta}^{s,u}_{b,b'}$. Actually, the domain
of the maps $\beta^{s,u}_{b^*}$ strictly contains the domain of the
maps $\overline{\beta}^{s,u}_{b,b'}$. Therefore the above relations
should be understood on the smaller domain, on which the maps
$\overline{\beta}^{s,u}_{b,b'}$ are actually defined. \\

Now, relation \eqref{eqn:center} is equivalent to $\overline{\b}_b(m) = (x_s, x_u)$ and $\b_{b^*}(m) = (x^*_s, x^*_u)$,
where: \begin{equation} \label{eqn:systems2}\begin{cases} x_s = \overline{\b}^u_{b}(x_u, m) \\ x_u = \overline{\b}^s_b(x_s,
m) \end{cases} \ \begin{cases} x_s^* = \b^u_{b^*}(x^*_u,m) \\
x_u^* = \b^s_{b^*}(x^*_s, m) \end{cases}\end{equation}

For fixed $m$, the solutions $(x_s, x_u)$ and $(x^*_s, x^*_u)$ are
fixed points of the contracting maps $\overline{\b}^s_b \circ \overline{\b}^u_b \times
\overline{\b}^u_b \circ \overline{\b}^s_b: Q^s \times Q^u \to Q^s \times Q^u$ and
$\b^s_{b^*} \circ \b^u_{b^*} \times \b^u_{b^*} \circ \b^s_{b^*}: Q^s
\times Q^u \to Q^s \times Q^u$, respectively. The contraction coefficient is $<<1$,
uniformly in $m$ and $b$. Therefore, the systems \eqref{eqn:systems2} have a unique
solution for each $m$. \\

As was shown in \eqref{eqn:aa} and \eqref{eqn:inc}, the maps $\beta^{s,u}_{b^*}$
and $\overline{\beta}^{s,u}_b$ of \eqref{eqn:systems2} are \ho continuous in $b$.
Therefore, the unique solutions of the systems \eqref{eqn:systems2} are also \ho
continuous in $b$, and thus so are the maps $\beta_{b^*}$ and $\overline{\beta}_b$.
Therefore, we have analogues of \eqref{eqn:aa}:

$$
||\overline{\b}_b - \b_{b^*}|| \le d(b,b^*)^{\a-O(\delta)}\cdot O(1),
\textrm{ }\textrm{ }\textrm{ }||\overline{\b}_{b'} - \b_{b^{*}}||
\le d(b',b^{*})^{\a-O(\delta)} \cdot O(1).
$$
By the triangle inequality, this implies:

$$
||\overline{\b}_b - \overline{\b}_{b'}|| \le (d(b,b^*)^{\a-O(\delta)}  + d(b',b^*)^{\a-O(\delta)})\cdot O(1) \le
$$
$$
\leq 2{(d(b,b^*) + d(b', b^*))}^{\a-O(\delta)}\cdot O(1) \le d(b,b')^{\a-O(\delta)} \cdot O(1),
$$
where the last inequality follows from \eqref{eqn:prod}. This proves the desired inequality in the chart $\ph_{b^*} \times
\textrm{Id}$ (that is, for the maps $\overline{\b}_{b}, \overline{\b}_{b'}:M
\longrightarrow Q^s \times Q^u$). On the manifold (that is, for the maps
$\widetilde{\b}_{b},\widetilde{\b}_{b'}:M \longrightarrow B$), the analogous
relation follows from the fact that the derivative of $\ph_{b^*}$ at $0$ is identity. \\

Therefore, relation \eqref{eqn:hol} is proved. Note,
that $O(1)$ above is a constant not depending on $b, b'$, but
depending on $\delta $ as in \eqref{eqn:hhol}. We have put $\rho $ in
the denominator of \eqref{eqn:inc} instead of $\delta $, because
$\rho = O(\delta )$. Finally, the inverse map $H^{-1}$ of
\eqref{eqn:homeo1} is explicitly given as $$
H^{-1}(b,m)=(\widetilde{\beta}_b(m),m). $$ This map is Lipschitz in
the variable $m$, and \ho continuous in the variable $b$ by
\eqref{eqn:hol}. Therefore $H^{-1}$ is \ho continuous. This
concludes the proof of Theorem~\ref{thm:mt}. \end{proof}

\section{\ho continuity of center-stable foliation} \label{sec:st}

In this section we complete the proof of Theorem~\ref{thm:st}.
Recall that in this theorem the map $h$ is a skew product itself,
see \eqref{eqn:CD}, whose  fibers are globally defined stable
manifolds for $h$. In this case, we will see that the central-stable
leaves of $\G$ can also be globally defined. \\

By analogy with Subsection~\ref{sub:lam}, for $z\in Z$ a
\emph{global central-stable leaf} is defined as a Lipschitz function
\begin{equation} \label{eqn:glob} \beta^s_z:F\times M \rightarrow Z, \end{equation} and its graph is defined as
$$ W^s_z=\gamma(\beta^s_z)=\{(\beta^s_z(f,m),f,m)| (f,m)\in F\times
M\}. $$ We ask that our leaves be Lipschitz close to the constant
function $z$, in the sense that: \begin{equation} \label{eqn:small2}
\max\left\{d(\beta^s_z,z)_{C^0}, \frac {\mLip \beta^s_z}{D}\right\}
\leq \frac {\delta}2. \end{equation} Finally, a \emph{global
central-stable lamination} is defined as a continuous assignment
$S^s=(\beta^s_z)=(W^s_z)$ of such leaves, as $z$ ranges over $Z$.
Such a lamination is called $\G-$invariant if \begin{equation}
\label{eqn:invfolcs} \G(W^s_z) = W^s_{\zeta(z)}, \textrm{ }\forall
z\in Z, \end{equation} where $D$ is so chosen that the estimates in the (sketch of
the) proof below work out. All these constructions are analogous to the
ones in Subsection~\ref{sub:lam}. Moreover, the entire machinery of
Lemma~\ref{lem:1} applies to our situation and produces a unique
$\G-$invariant lamination $S^s=(\beta^s_z)$ satisfying
\eqref{eqn:small2} for $D$ properly chosen. We will henceforth focus
solely on this lamination. In particular, since $\zeta$ is expanding
we obtain: \begin{equation} \label{eqn:hold0} d(\beta^s_z
,\beta^s_{z'})_{C^0} \le \frac
{d(z,z')^{\alpha-O(\rho)}}{O(\rho)^{\alpha}}, \textrm{ }\textrm{
}\textrm{ where }\alpha=\frac {\ln \mu}{\ln \mu_-}. \end{equation}
This is proven in analogous fashion to statement b) of
Theorem~\ref{thm:mt}, which was proved in the previous Section. \\

\begin{proof} \textbf{of Theorem~\ref{thm:st}:}

\begin{Prop}
\label{prop:whole phase space}
The leaves $W^s_z=\emph{Graph}(\beta^s_z)$ are disjoint and they cover the whole of $X$:
$$
X=\bigsqcup_{z\in Z} W^s_z.
$$
\end{Prop}

\begin{proof} The fact that the leaves are disjoint is proven analogously to
Proposition~\ref{prop:disjointfibers}. As for their union being the whole of $X$,
this is equivalent to the following claim: for any $z\in Z$ and $y\in F\times M$,
there exists $\widetilde{z}\in Z$ such that $\beta^s_{\widetilde{z}}(y)=z$. Let us fix $y$ and $z$, and prove this claim. \\

Fix a coordinate neighborhood of radius $2\delta$ of $z$ inside $Z$. Let
$D(z,\delta)$, $D(z,2\delta)$, $S(z,\delta)$, $S(z,2\delta)$ be the
balls/spheres centered at $z$ of radii $\delta$ and $2\delta $ in
$Z$, respectively. The map $f:D(z,\delta) \rightarrow D(z,2\delta)$
given by $f(\widetilde{z}) = \beta_{\widetilde{z}}(y)$ is
well-defined, because \eqref{eqn:small2} implies that
$d(f(\widetilde{z}),\widetilde{z})\leq \delta/2$. Moreover,
\eqref{eqn:hold0} implies that the map $f$ is continuous. Therefore,
sliding points along a straight line segment gives us an isotopy
between the identity map of $D(z,\delta)$ and $f$: $$h_t(\tilde z,y)
= (\tilde z,y) + t((\b_{\tilde z}(y) - \tilde z),0)$$
Because of $d(f(\widetilde{z}),\widetilde{z})\leq \delta/2$, the
image of the boundary sphere $h_t(S(z,\delta))$ never touches the
center $z$ during this isotopy. Therefore, the index of $z$ with
respect to the sphere $h_t(S(z,\delta))$ does not change during the
isotopy. Therefore $$ z\in \textrm{Im}(f) \Rightarrow \exists
\widetilde{z} \textrm{ such that } \beta^s_{\widetilde{z}}(y)=z. $$
\end{proof}

By Proposition~\ref{prop:whole phase space}, the map $q:X\rightarrow Z$
given by sending $W^s_z$ to $z$ is well-defined. Moreover, the $\G-$invariance condition
\eqref{eqn:invfolcs} implies that $q$ makes the diagram~\eqref{eqn:semisol} commute.
Part $b)$ of Theorem~\ref{thm:st} follows immediately from \eqref{eqn:small2} and \eqref{eqn:hold0}. \\

Finally, let us prove the relation $q|_Y=\pi\circ p$. Take any point
$b=(f,z)\in F \times Z$, and recall that we denote $z=\pi(b)$. If we
take the map $\beta^s_{\pi(b)}$ defining the global lamination (see
\eqref{eqn:glob}), and restrict it to the $\delta$ neighborhood of
$f\in F$, we obtain a map $\overline{\beta}_b^s$ as in
\eqref{eqn:leaf}. In other words restricting the leaves of the
global lamination $W^s_{\pi(b)}$ produces a valid local lamination
$\overline{W}^s_b$. Since the global lamination $W^s_{\pi(b)}$ is
$\G-$invariant, it is easily seen that the local lamination
$\overline{W}^s_b$ will also be $\G-$invariant. \\

But local laminations are unique, as proved in
Corollary~\ref{cor:fixed}!  Therefore, the local leaves
$\overline{W}^s_b$ coincide with the central-stable leaves $W^s_b$
of Section~\ref{sec:grtrans}. By the very definition of
$\overline{W}^s_b$, this implies that $W^s_b \subset W^s_{\pi(b)}$,
for all $b$. Since $W_b \subset W^s_b$ by construction, we conclude
that:

$$ W_b \subset W^s_{\pi(b)}, \textrm{ }\forall \textrm{ }b. $$ Now
take any point $x\in Y = \bigsqcup_{b \in \Lambda} W_b$, and assume
$x\in W_b$. By the definition of $p$, we have $p(x)=b$. But the
above inclusion implies that $x\in W^s_{\pi(b)}$, and then by the
definition of $q$, we have $q(x)=\pi(b)$. This precisely amounts to
saying that $q|_Y=\pi\circ p$. \end{proof}

\section{Fubini regained}
\label{sec:fubini}

In this section we prove Theorem~\ref{thm:fubini}. Moreover, at the
end of this Section we discuss the ``weak ergodic theorem'' that
appears in \cite{IKS08}.

\subsection {Measure zero and incomplete Hausdorff dimension}

Let us begin by recalling the concept of Hausdorff dimension,
denoted henceforth by $\textrm{dim}_H$.

\begin{Def} \label{def:haus} Let $A$ be a subset of a Euclidean
space. A \emph{cover} $U$ of $A$ is a finite or countable collection
of balls $Q_j$ of radii $r_j$ whose union contains $F$. The
$d$-dimensional volume of $U$, denoted by $V_d(U)$, is defined as $$
V_d(U) = \sum_j r_j^{d_j}. $$ The \hd of $A$ is defined as the
infimum of those $d$ for which there exists a  cover of $A$ with
arbitrarily small $d-$dimensional volume: $$ \emph{dim}_H\textrm{ }A
= \emph{inf} \{ d|\forall \e > 0 \textrm{ }\exists \emph{ a cover
}U\emph{ of }A\emph{ such that } V_d(U) < \e \} . $$ \end{Def}

Note that a compact manifold of dimension $d$ also has Hausdorff
dimension $d$. The same holds for a set of a positive Lebesgue
measure on the Riemannian \man of dimension $d$.
Theorem~\ref{thm:fubini} immediately follows from the following two
propositions:

\begin{Prop}
\label{prop:fubini1}
Recall the general setup of Theorem~\ref{thm:st}. If $A\subset Z$ satisfies
$$
\emph{dim}_H \textrm{ }A < \frac {\ln \mu}{\ln \mu_-}\cdot \emph{dim } Z,
$$
then for $\rho$ small enough, the set $q^{-1}(A)$ has Lebesgue measure 0 in $X$.
\end{Prop}

Now recall the particular setup of Theorem~\ref{thm:fubini}, which takes place over the
solenoid map. Note that in this case we have $Z=S^1$ and $\mu_-=\mu=\frac 12$.

\begin{Prop}
\label{prop:fubini2}
For any $\kappa>0$ and finite word $w$, there exists $\e=\e(\kappa,w)$ such that the
set $A_{\kappa,w}$ of Subsection~\ref{sub:fubini} has Hausdorff dimension at most $1-\e$.
\end{Prop}

\subsection{Saving Fubini: the proof of Proposition~\ref{prop:fubini1}}

We are in the more general setup of Theorem~\ref{thm:st}. Since
$X=Z\times F \times M$, the classical Fubini theorem states that $$
\textrm{mes}(q^{-1}(A)\cap Z\times \{x\})=0, \forall x\in F \times M
\Rightarrow \textrm{mes}(q^{-1}(A))=0. $$ So all we need to do is to
show that for any fixed $x\in F \times M$, the intersection
$q^{-1}(A)\cap Z\times \{x\}$ has measure 0 in $Z$. By the very
definition of the map $q$ of \eqref{eqn:semisol}, this intersection
is nothing but the set $\{\beta^s_z(x)|z\in A\}\subset Z$.
Moreover, by statement b) of Theorem~\ref{thm:st} the map $$
\ph:Z\rightarrow Z, \textrm{ }\textrm{ }\textrm{ } \ph(z) =
\beta^s_z(x) $$ is \ho continuous with exponent $\alpha=\frac {\ln
\mu}{\ln \mu_-}-O(\rho)$. All that we need to prove is that the set
$\ph(A)$ has measure 0 in $Z$. The following general lemma will do
the trick:

\begin{Lem}[Falconer] \label{lem:haus} Let $Z$ be any Riemannian
manifold, and $A \subset Z$ a subset. If $\ph:Z\rightarrow Z$ is a
\ho map with exponent $\a$, then $$ \emph{dim}_H\textrm{ }\ph (A)
\le \frac {\emph{dim}_H \textrm{ }A} \a $$ \end{Lem}

The proof of this Lemma can be found in \cite{Fa}; the proof is
straightforward. The above Lemma and the assumptions of
Proposition~\ref{prop:fubini1} imply that for small enough $\rho$,
we will have $\textrm{dim}_H\textrm{ }\ph(A)<\textrm{dim }Z$.
Therefore, $\ph(A)$ has Lebesgue measure 0 in $Z$, and as we have
seen above this implies that $q^{-1}(A)$ has Lebesgue measure 0 in
$X$. This concludes the proof of Proposition~\ref{prop:fubini1}.

\subsection{Large deviations: the proof of Proposition~\ref{prop:fubini2}}

In this section, we must prove that for any $\kappa>0$ and finite
word $w$, the set $A_{\kappa,w}\subset S^1$ of
Subsection~\ref{sub:fubini} has Hausdorff dimension at most $1-\e$.
Call a finite word of length $N$ a $\kappa,w-$atypical word if the
frequency of appearances of $w$ in that word is outside the interval
$[2^{-n}-\kappa, 2^{-n}+\kappa]$. Obviously, if a sequence is
$\kappa,w-$atypical then infinitely many of its initial parts will
be $\kappa,w-$atypical words. Thus for any $N_0$, we have the
following inclusion: $$ \{\kappa,w-\textrm{atypical sequences}\}
\subset \bigcup^{N\geq N_0}\bigcup^ {\textrm{length }v= N}_{v
\textrm{ is }\kappa,w-\textrm{atypical}} \{\textrm{sequences
starting with }v\}. $$ Looking at the points of $S^1$ that
correspond in binary notation to these sequences, we have: $$
A_{\kappa,w} \subset \bigcup^{N\geq N_0}\bigcup^{\textrm{length
}v=N}_{v \textrm{ is }\kappa,w-\textrm{atypical}} \{\textrm{ball of
radius }2^{-N} \textrm{ around }\overline{0.v}\}. $$ This produces a
covering $U$ of the set $A_{\kappa,w}$, as in
Definition~\ref{def:haus}. Let us compute the $1-\e$ dimensional
volume of this covering: $$ V_{1-\e}(U)\leq \sum_{N\geq N_0}
2^{-N(1-\e)}\cdot \#\{\kappa,w-\textrm{atypical words of length
}N\}. $$ By Theorem~\ref{thm:largedev} below, we can estimate the
number of $\kappa,w-$atypical words of length $N$, thus obtaining $$
V_{1-\e}(U) \leq \sum_{N \geq N_0} 2^{-N(\nu-\e)}=\frac
{2^{-N_0(\nu-\e)}}{1-2^{\e-\nu}}. $$ If we choose $\e<\nu$ and let
$N_0\rightarrow \infty$, the above expression can be made
arbitrarily small. Therefore, the Hausdorff dimension of the set
$A_{\kappa,w}$ is at most $1-\e$. This concludes the proof of
Proposition~\ref{prop:fubini2} and of Theorem~\ref{thm:fubini},
modulo the following estimate:

\begin{Thm}[Large Deviation Theorem, \cite{V05}]
\label{thm:largedev} There exists $\nu=\nu(\kappa,w)$ such that for
any $N$ greater than some $N_0$, the number of $\kappa,w-$ atypical
words of length $N$ is at most $2^{N(1-\nu)}$. \end{Thm}


\subsection{Weak ergodic theorem}   \label{subsec:weak}

For the sake of completeness, we formulate here the weak ergodic theorem of
\cite {IKS08}, whose proof is very closely related to the above material. \\

The classical ergodic theorem claims that for a given ergodic map
and a continuous function $\ph $, the set of points for which the
time average of $\ph $ either does not exist or is not equal to the
space average of $\ph $, has measure zero. Here we claim that, for
the duplication of a circle, for any fixed continuous function $\ph
\in C(S^1)$ and any $\delta > 0$, the set of points for which the
sequence of partial time averages of $\ph $ has a limit point that
differs from the space average of $\ph $ by more than $\delta$, has
Hausdorff dimension smaller than $1$. We expect that this theorem
may be generalized to any ergodic hyperbolic map of a compact
Riemannian manifold $M^n$ with a smooth invariant measure.

\begin{Thm}   \label{thm:weak} Let $\zeta$ be the duplication of the
circle $S^1 = \rr /\zz$, given by $\zeta(y)=2y$. Let $\ph \in
C(S^1)$ and $\d  > 0$ be given. The partial time averages of $\ph$
and its space average are defined as: $$ \ph_n(y) = \frac 1n
\sum_{i=0}^{n-1}\ph (\zeta^i(y)), \textrm{ }\textrm{ }\textrm{ } I =
\int_{S^1}\ph. $$ Then the set $$ K_{\ph ,\d } = \{ y| \textrm{ the
sequence } \ph_n(y) \textrm{ has a limit point outside } [I - \d , I
+ \d ] \} $$ has Hausdorff dimension smaller than 1. \end{Thm}

A similar theorem for Anosov \diffeos of the two-torus was proved
recently by Saltykov \cite{S09}.

\section{Appendix}
\label{sec:app}

\subsection{The graph transform map made explicit}

Recall that the graph transform map $\g_b:\B^s \rightarrow \B^s$ was defined by $\b^s \rightarrow \overline{\b}^s$, where:

$$ \{\G^{-1}(\ph_{h(b)}(x_s,\beta^s(x_s,m)),m)\} \supset
\{(\ph_b(x_s,\overline{\beta}^s(x_s,m)),m)\}. $$ We now want to turn
this implicit definition into an explicit formula. Recall our
notation $\gamma(\beta)$, under which the above becomes: $$
\textrm{Im } \G^{-1}\circ (\ph_{h(b)}\times \textrm{Id}) \circ
\gamma(\beta^s) \supset \textrm{Im } (\ph_{b}\times
\textrm{Id})\circ \gamma(\overline{\beta}^s). $$ If we write $\G_b
={(\ph_{h(b)}(C\delta)\times \textrm{Id})}^{-1} \circ \G \circ
(\ph_b(\delta) \times \textrm{Id})$ as in \eqref{eqn:hcoord}, then
our relation takes the form:

\begin{equation} \label{eqn:implicit1} \textrm{Im }\G_b^{-1} \circ
\gamma(\beta^s) \supset \textrm{Im } \gamma(\overline{\beta}^s).
\end{equation} Write $\pi_u:Q^s \times Q^u \times M \rightarrow Q^u$
and $\pi_{sc}:Q^s\times Q^u\times M \rightarrow Q^s \times M$ for
the standard projections, and define:

\begin{equation} \label{eqn:Gbb} G_{\overline\b ,b} = \pi_{sc} \circ
\G_b \circ \ga (\overline{\b}^s):Q^s \times M \rightarrow Q^s \times
M. \end{equation} Then \eqref{eqn:implicit1} is equivalent to:

\begin{equation} \label{eqn:graph} \pi_u \circ \G_b^{-1} \circ
\gamma (\b^s) \circ G_{\overline\b, b} = \overline{\b}^s. \end{equation}

\begin{Prop} \label{prop:composition} The composition
\eqref{eqn:Gbb} is well-defined and $$ \emph{Lip }
G_{\overline\beta,b}\leq (L + O(\delta))\cdot (1+\emph{Lip
}\overline\b^s), $$ where $L$ is the constant from
Definition~\ref{def:domsplit}. A similar estimate holds in the
central-unstable case. \end{Prop}

\begin{proof} Define the composition $$ F_{0 ,b} = \pi_{sc} \circ
\F_b \circ \ga (0), $$ in analogy with \eqref{eqn:Gbb}, with $\G$
replaced by $\F$ and $\overline{\beta}^s$ replaced by the zero map
$0:Q^s\times M \rightarrow Q^u$. Since $d(\G,\F)_{C^1}\leq \rho$, we
see that: \begin{equation} \label{eqn:qqq} \mLip G_{\beta,b} \leq
(\mLip F_{0,b}+O(\rho))\cdot (1+\textrm{Lip }\overline\b^s)
\end{equation} But one can simply unravel the definition of
$F_{0,b}$ when $\F$ is a skew product, and obtain $$
F_{0,b}(x_s,m)=(\pi_s\circ h_b(x_s,0), f_{(x_s,0)}(m)). $$ From this
it is clear that $$ \mLip F_{0,b} \leq L + O(\delta). $$ Recalling
that we always choose $\delta=O(\rho)$, \eqref{eqn:qqq} implies: $$
\mLip G_{\overline{\beta},b} \leq (L + O(\delta))\cdot (1+\textrm{Lip
}\overline\b^s). $$ \end{proof}

\begin{Prop}
\label{prop:last}
For any two central-stable leaves $\overline\b^s_0, \overline\b^s_1 \in \B^s$, we have:
$$
||G_{\overline\b_0,b} - G_{\overline\b_1,b}|| \leq O(1)\cdot ||\overline\beta^s_0-\overline\beta^s_1||.
$$
A similar result holds in the central-unstable case.
\end{Prop}

\begin{proof} We have: $$ ||G_{\overline\b_0,b} -
G_{\overline\b_1,b}||=||\pi_{sc}\circ \G_{b} \circ
\ga(\overline\b^s_0) - \pi_{sc}\circ \G_{b} \circ
\ga(\overline\b^s_1)||\leq $$ $$ \leq \textrm{Lip }(\pi_{sc}\circ
\G_{b})\cdot ||\ga(\overline\b^s_0) - \ga(\overline\b^s_1)|| \leq
O(1)\cdot ||\overline\beta^s_0-\overline\beta^s_1||. $$

\end{proof}

\subsection{Persistence of \ho skew products}

The second, independent technical result that we will prove concerns
the setup of Theorem~\ref{thm:mt}: we have a small
$\rho-$perturbation $\G$ of the skew product $\F$ from
Theorem~\ref{thm:mt}. This theorem tells us that $\G$ is conjugated
to a skew product $G$: $$ G(b,m)=(h(b), g_b(m)). $$ In this
Subsection, we will prove formulas \eqref{eqn:close} and
\eqref{eqn:holfibermaps}. To this end, from the very definition of
$G$ we have the following explicit formula for the fiber maps $g_b$:
\begin{equation} \label{eqn:fibermapg} g_b(m)=\pi_m
(\G(\widetilde{\beta}_b(m),m)), \textrm{ }\textrm{ }\textrm{ }
g_b^{-1}(m)=\pi_m(\G^{-1}(\widetilde{\beta}_{h(b)}(m),m)) \end{equation} where
$\pi_m:X=B\times M \rightarrow M$ is the standard projection.
Obviously, we have $$ f_b(m)=\pi_m(\F(b,m)), \textrm{ }\textrm{
}\textrm{ } f_b^{-1}(m)=\pi_m(\F^{-1}(h(b),m)). $$ Since $d(\G^{\pm
1}, \F^{\pm 1})_{C^1}<\rho$, it follows from the above formulas that
$$ d(g_b,f_b)_{C^1} \leq d(\G(\widetilde{\beta}_b(m),m), \F(b,m))_{C^1} \leq $$
$$ \leq d(\G(\widetilde{\beta}_b(m),m), \G(b,m))_{C^1}+\rho \leq
||\G||_{C^1}\cdot d(\widetilde{\beta}_b,b)_{C^1}+\rho = O(\rho), $$ and
similarly for $d(g_b^{-1}, f_b^{-1})_{C^1}$. This proves
\eqref{eqn:close}. As for the \ho property, we have that $$ d(g_b,
g_{b'})_{C^0}\leq ||\G||_{C^1} \cdot d(\widetilde{\beta}_b, \widetilde{\beta}_{b'})_{C^0}
\leq O(d(b,b')^\a), $$ by \eqref{eqn:hol}. The statement concerning
$d(g_b^{-1}, g_{b'}^{-1})_{C^0}$ is proved analogously, thus
concluding the proof of \eqref{eqn:holfibermaps}.

\subsection{Acknowledgments}

We are grateful to A.Bufetov, A.Klimenko, C.Pugh, M.Shub for
fruitful discussions. The second author would like to thank the Max
Planck Institut f\"ur Mathematik in Bonn for hosting him while a
significant portion of this work was written.


\begin{thebibliography} {2}





\bibitem[1]{Bo} C. Bonatti, L. Diaz, M. Viana, Dynamics beyond uniform hyperbolicity (Encyclopaedia of mathematical sciences, vol. 102),    Springer, 2004


\bibitem[2]{BM07} A. Bonifant, J. Milnor, Schwarzian derivatives and
    cylinder maps, preprint

\bibitem[3]{DG} L. Diaz, A. Gorodetski, Non-hyperbolic ergodic
    measures for non-hyperbolic homoclinic classes, to appear

\bibitem[4]{Fa} K. Falconer, Fractal  geometry: mathematical
    foundations and applications, John Wiley and Sons, USA, 1990


\bibitem[5]{GI99}  A. S. Gorodetskii, Yu. S. Ilyashenko, Some new
    robust properties of invariant sets and attractors of dynamical
    systems. (Russian) Funktsional. Anal. i Prilozhen. 33 (1999),
    no. 2, 16--30, 95; translation in Funct. Anal. Appl. 33 (1999),
    no. 2, 95--105


\bibitem[6]{GI00} A. S. Gorodetskii, Yu. S. Ilyashenko, Some
    properties of skew products over a horseshoe and a solenoid.
    (Russian)  Tr. Mat. Inst. Steklova  231  (2000),  Din. Sist.,
    Avtom. i Beskon. Gruppy, 96--118; translation in  Proc. Steklov
    Inst. Math. 2000,  no. 4 (231), 90--112

\bibitem[7]{GIKN} A.Gorodetski, Yu. Ilyashenko, V.Kleptsyn,
    M.Nalski, Non-removable zero Lyapunov exponents, Functional
    Analysis and Appl., vol 39,  no 1, 2005, 27--38

\bibitem[8]{G06} A. S. Gorodetskii, Regularity of central leaves of
    partially hyperbolic sets and applications. (Russian) Izv. Ross.
    Akad. Nauk Ser. Mat. 70 (2006), no. 6, 19--44; translation in
    Izv. Math. 70 (2006), no. 6, 1093--1116

\bibitem[9]{HPS} M. W. Hirsch, C. C. Pugh, M. Shub, Invariant manifolds (Lecture Notes in Mathematics, vol. 583), 1977,
    ii+149

\bibitem[10]{I08} Yu. Ilyashenko, Diffeomorphisms with intermingled
    attracting basins, Functional Analysis and Appl.,  42 (2008), no
    4, 60--71

\bibitem[11]{IKS08} Yu. Ilyashenko, V. Kleptsyn, P. Saltykov,
    Openness of the set of boundary preserving maps of an annulus with
    intermingled attracting basins, JFPTA, 2008, v.
3, 449--463

\bibitem[12]{IN} Yu. Ilyashenko, A. Negut, Invisible parts of
    attractors,  Nonlinearity, v. 23 (2010) 1199-1219.

\bibitem[13]{IV} Yu. Ilyashenko, D. Volk, Cascades and
    $\e$-invisibility, accepted to Journal of the fixed point theory and applications

\bibitem[14]{K94} I. Kan, Open sets of difeomorphisms having two
    attractors, each with everywhere dense basin, Bull. Amer.
    Math. Soc,   v. 31  (1994), 68--74


\bibitem[15]{KS*} V. Kleptsyn, P. Saltykov, $C^2$-stable example of
    attractors with intermingled basins for boundary preserving maps
    of an annulus, in preparation.

\bibitem[16]{M75} R. Mane, Hyberbolic sets for semilinear
    parabolic equations. Bol. Soc. Brasil. Mat. 6 (1975), no. 2,
    145--153.


\bibitem[17]{M97} J. Milnor, Fubini foiled: Katok's paradoxical
    example in measure theory. Math. Intelligencer 19 (1997), no. 2,
    30--32.

\bibitem[18]{N}  V. Nitica, A. T\"or\"ok,  Cohomology of
    dynamical    systems and rigidity of partially hyperbolic actions of
    higher-rank lattices, Duke Math. J.   79 (1995), issue    3


\bibitem[19]{P04} Ya. Pesin, Lectures on partial hyperbolicity and
    stable ergodicity, Eur. Math. Soc., 2004


\bibitem[20]{RW01}  D. Ruelle, A. Wilkinson, Absolutely
    singular dynamical foliations. Comm. Math. Phys. 219 (2001), no.
    3, 481--487.

\bibitem[21]{S09} P. Saltykov, Special ergodic theorem for the Anosov
    \diffeo of     a    two-torus, submitted

\bibitem[22]{SW00} M. Shub, A. Wilkinson, Pathological
    foliations and removable zero exponents. Invent. Math. 139
    (2000), no. 3, 495--508.

\bibitem[23]{SW00a} M. Shub, A. Wilkinson, Stably ergodic
    approximation: two examples. Ergodic Theory Dynam. Systems 20
    (2000), no. 3, 875--893.

\bibitem[24]{V05} S.R.S. Varadhan, Large Deviations, The
    Annals of Probaility, 2008, Vol 36, no. 2, 397-419.

\end{thebibliography}
\end{document}